\documentclass[11pt]{article}
\usepackage{a4}
\usepackage{amsmath}
\usepackage{amssymb}
\usepackage{amsthm}
\usepackage{graphicx}
\usepackage{amsfonts}

\newtheorem{theorem}{Theorem}[section]
\newtheorem{proposition}[theorem]{Proposition}

\newtheorem{question}[theorem]{Question}
\newtheorem{conjecture}[theorem]{Conjecture}
\newtheorem{lemma}[theorem]{Lemma}

\newcommand{\R}{\mathbb{R}}
\newcommand{\N}{\mathbb{N}}
\newcommand{\K}{{\mathbf K}}
\newcommand{\J}{{\mathbf J}}
\newcommand{\T}{{\mathbf T}}
\renewcommand{\L}{{\mathbf L}}
\newcommand{\LL}{\mathcal{L}}
\newcommand{\F}{\mathcal{F}}

\newcommand{\C}{\mathcal{C}}
\newcommand{\G}{\mathcal{G}}
\newcommand{\A}{\mathcal{A}}
\newcommand{\B}{\mathcal{B}}
\newcommand{\D}{\mathcal{D}}
\newcommand{\X}{\mathcal{X}}
\newcommand{\Y}{\mathcal{Y}}
\newcommand{\Z}{\mathcal{Z}}

\newcommand{\heading}[1]{\medskip\par\noindent{\bf #1}}

\title{A counterexample to Wegner's conjecture on good covers}
\author{Martin Tancer\thanks{Department of Applied Mathematics and Institute for Theoretical Computer
Science (supported by project 1M0545
of The Ministry of Education of the Czech Republic), Faculty of Mathematics
and Physics, Charles University, Malostransk\'e n\'am.~25, 118~00 Prague,
Czech Republic. Partially supported by project GAUK 49209. {\tt Email:tancer@kam.mff.cuni.cz}.
}}

\begin{document}
\maketitle
\begin{abstract}
In 1975 Wegner conjectured that the nerve of every finite good cover in $\R^d$
is $d$-collapsible. We disprove this conjecture.

A good cover is a collection of open sets in $\R^d$ such that the intersection
of every subcollection is either empty or homeomorphic to an open $d$-ball. A
simplicial complex is $d$-collapsible if it can be reduced to an empty complex
by repeatedly removing a face of dimension at most $d-1$ which is contained in a unique maximal face.
\end{abstract}

\section{Introduction}
In 1975~Wegner~\cite{wegner75} introduced $d$-collapsible simplicial complexes.
His definition comes from studying intersection patterns of convex sets. He
proved that simplicial complexes coming from finite collections of convex sets
(as their nerves) are $d$-collapsible. He also conjectured that his result has
a topological extension when collections of convex sets are replaced by good
covers. The purpose of this article is to disprove this conjecture.

We assume that the reader is familiar with simplicial complexes, otherwise we
refer to introductory chapters of books
like~\cite{hatcher01,matousek03,munkres84}. 

\heading{Good covers.} 
A \emph{$d$-cell} is an open subset of $\R^d$ homeomorphic to a ball. A
\emph{good cover} in $\R^d$ is a collection $\G$ of $d$-cells in $\R^d$ such
that for every subcollection $\G' \subseteq \G$ the intersection of the sets in
$\G'$ is either a $d$-cell or empty.\footnote{Different notions of a good cover
appear in the literature. Sometimes the cells are assumed being compact instead
of open. The intersection of $\G'$ can be also assumed contractible instead of
being a $d$-cell. For our purposes there is not a big difference while
constructing a counterexample. It can be easily modified to fulfill the above
mentioned conditions.} We remark that a collection of open convex sets in
$\R^d$ always form a good cover.

\heading{$d$-representable complexes.}
A \emph{nerve} of a collection $\F$ of sets in $\R^d$ is a simplicial complex
whose vertices are the sets in $\F$ and whose faces are collections of vertices
with a nonempty intersection. A finite simplicial complex is \emph{convexly
$d$-representable} if it is isomorphic to a collection of convex sets in
$\R^d$; it is \emph{topologically $d$-representable} if it is isomorphic to a
good cover in $\R^d$. Standard notion appearing in the literature is
\emph{$d$-representable} instead of convexly $d$-representable. However, we add
`convexly' in order to clearly distinguish convex sets and good covers. 

\heading{$d$-collapsible complexes.}
Let $\K$ be a simplicial complex. Assume that $\sigma$ is a face of $\K$ of
dimension at most $d-1$ such that there is only one maximal face of $\K$
containing $\sigma$. Then we say that $\sigma$ is $d$-collapsible and that the
complex
$$
\K' := \K \setminus \{\eta \in \K: \sigma \subseteq \eta\}
$$
arises from $\K$ by an \emph{elementary $d$-collapse}.
A simplicial complex $\K$ is \emph{$d$-collapsible} if it can be reduced to an
empty complex by a sequence of elementary $d$-collapses.

Wegner~\cite{wegner75} proved that every convexly $d$-representable simplicial
complex is $d$-collapsible. He also conjectured that his result has the
following topological extension.

\begin{conjecture}[Wegner, 1975]
\label{c:weg}
Every topologically $d$-representable simplicial complex is $d$-collapsible.
\end{conjecture}
We disprove this conjecture for every $d \geq 2$. We start with constructing a
simplicial complex $\L$ which is topologically $2$-representable but not
$2$-collapsible. In higher dimensions we obtain a counterexample by using
suspensions of the complex $\L$. 

\begin{theorem}
\label{t:main}
For every $d \geq 2$ there is a simplicial complex which is topologically
$d$-representable but not $d$-collapsible.
\end{theorem}

\heading{Additional background.}
Conjecture~\ref{c:weg} is a natural question in the context of Helly-type
theorems. For example the Helly theorem~\cite{helly23} can be formulated in
such a way that a convexly $d$-representable simplicial complex containing all
$d$-faces has to be already a full simplex. Once we know that convexly
$d$-representable simplicial complexes are $d$-collapsible, it is a simple
consideration to prove the Helly theorem\footnote{There is also a simple
geometric proof of the Helly theorem, of course.}. In this case, the Helly
theorem has a topological extension for topologically $d$-representable
simplicial complexes~\cite{helly30}. Many Helly~type theorems have a similar
topological extension; see, e.g., the introduction of~\cite{matousek-tancer09}
for more detailed list. However, the conclusion for $d$-collapsible complexes may be stronger than for topologically $d$-representable complexes, e.g.,
in~\cite{kalai-meshulam05}. Our counterexample thus shows that results for
$d$-collapsible complexes cannot be generalized all at once for topologically
$d$-representable complexes.

\section{Planar case}
We start this section with describing the complex $\L$. Let $A_1, A_2, A_3,
B_1, B_2, B_3,$ $C_1, C_2, C_3, D, X_1, X_2, X_3, Y_1, Y_2, Y_3, Z_1, Z_2, Z_3$
be the (open) sets from the Figure~\ref{f:ce}. We also set $\A := \{A_1, A_2, A_3\}$,
$\B := \{B_1, B_2, B_3\}$, $\C := \{C_1, C_2, C_3\}$, $\D := \{D\}$, $\X :=
\{X_1, X_2, X_3\}$, $\Y := \{Y_1, Y_2, Y_3\}$, and $\Z := \{ Z_1, Z_2, Z_3\}$.
Let $\LL$ be the collection of all these sets, i.e., $\LL := \A \cup \B \cup \C
\cup \D \cup \X \cup \Y \cup \Z$. Finally, $\L$ is the nerve of $\LL$.

\begin{figure}
\begin{center}
\label{f:ce}
\includegraphics{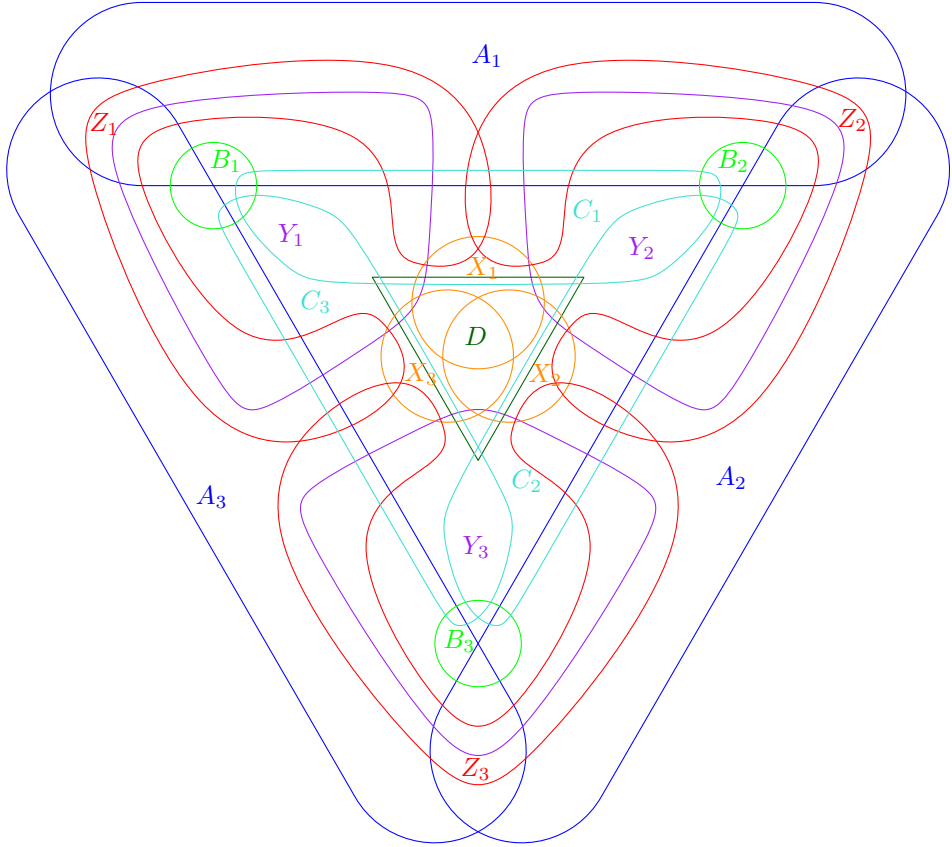}
\caption{The sets $A_1, \dots, Z_3$. We rather supply more detailed description
of the sets if the picture is print only in black and white: The sets $A_*$ are
the ovals on the boundary; $B_*$ are the small discs close to the
boundary; $C_*$ are the bread-shaped sets; $D$ is the triangle in the center;
$X_*$ are the circles close to the center; $Y_*$ are the bell-shaped sets; and $Z_*$ are the boomerang-shaped sets}
\end{center}
\end{figure}

We will show that $\L$ is topologically $2$-representable but not
$2$-collapsible.

\subsection{Topological representability}
It is sufficient to show that $\LL$ is a good cover. This property can be
hand-checked; however, we offer an alternative approach.

First we realize that all sets of $\LL \setminus \Z$ are convex. Thus $\LL
\setminus \Z$ is a good cover. It remains to check that adding sets of $\Z$
does not violate this property.

Let $Z \in \Z$ and let $\LL^{Z} := \{L \cap Z: L \in \LL\}$. We are done as
soon as we show that $\LL^{Z_1}, \LL^{Z_2}$, and $\LL^{Z_3}$ are good covers.

Because of the symmetry we show it only for $\LL^{Z_1}$. The sets of
$\LL^{Z_1}$ can be transformed into convex sets by a homeomorphism of $\R^2$.
See Figure~\ref{f:hom}. Thus they form a good cover.

\begin{figure}
\begin{center}
\label{f:hom}
\includegraphics{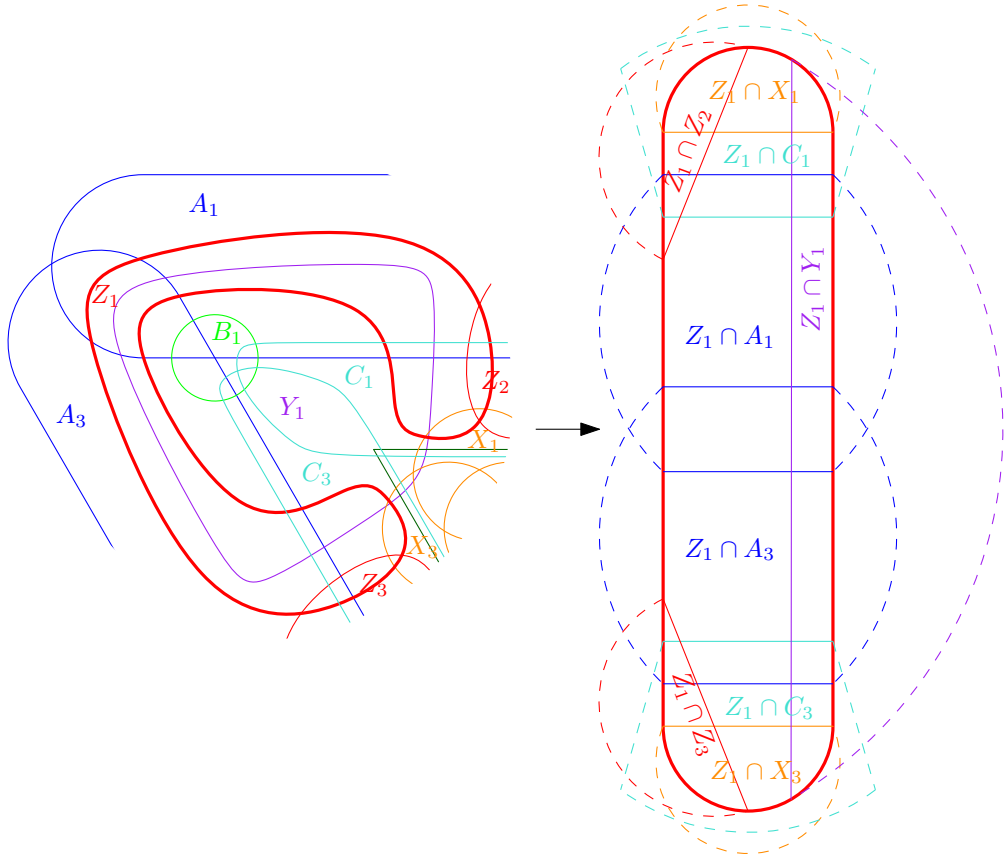}
\caption{A transformation of $\LL^{Z_1}$. Whatever is outside of $Z_1$ can be
ignored.}
\end{center}
\end{figure}

\subsection{Non-collapsibility by case analysis}
Here we prove that $\L$ is not 2-collapsible by case analysis. We get a bit
stronger results that will help us for higher dimensions. Disadvantage of
this proof is that it does not give an explanation how is the complex
constructed. Therefore we supply an additional heuristic 
explanation in the next
subsection, although it would need a bit more effort to turn that explanation
into a proof.

For a simplicial complex $\K$ we set 
$$
\gamma_0(\K) := \min\{d: \K \hbox{ has a $d$-collapsible face} \}.
$$
The fact that $\L$ is not 2-collapsible is implied by the following
proposition.

\begin{proposition}
\label{p:gamma}
$\gamma_0(\L) = 3$.
\end{proposition}

In order to prove the proposition we need a simple lemma.
\begin{lemma}
\label{l:split}
Let $\K$ be a simplicial complex and $\sigma$ be a 1-face (edge) of it. 
Assume that $u$ and $v$ are vertices of $\K$ not belonging to $\sigma$ such that
$\sigma \cup \{u\} \in \K$, $\sigma \cup \{v\} \in \K$, but $\sigma \cup
\{u,v\} \not \in \K$. Then $\sigma$ is not a $2$-collapsible face of $\K$.
\end{lemma}

\begin{proof}
If $\tau$ is a unique maximal face of $\K$ containing $\sigma$ then $u, v \in
\tau$ due to the conditions of the lemma. However, $\sigma \cup \{u, v\} \not
\in \K$.
\end{proof}

\begin{proof}[Proof of Proposition~\ref{p:gamma}]
In the spirit of Lemma~\ref{l:split} for every 1-face $\sigma \in \L$ we find a
couple of vertices $u, v \in \L$ such that $\sigma \cup \{u\}, \sigma \cup\{v\}
\in \L$, but $\sigma \cup \{u,v\} \not \in \L$. It is sufficient to check
1-faces since if a 0-face (vertex) $w$ is 1-collapsible then any 1-face containing $w$
is 1-collapsible as well. Moreover, it is sufficient to check only some 1-faces
because of the symmetries of the complex. The rest of the proof is given by the
following table.
\begin{center}
\begin{tabular}{cccccccc}
$\sigma$ & $u, v$ & & $\sigma$ & $u,v$ & & $\sigma$ & $u,v$\\
$\{A_1, A_2\}$ & $B_2, Z_2$ & & $\{A_1,B_1\}$ & $C_1, A_3$ & & $\{A_1,C_1\}$ & 
$B_1, B_2$\\
$\{A_1, Y_1\}$ & $B_1, Z_1$ & & $\{A_1,Z_1\}$ & $C_1, A_3$ & & $\{B_1,C_1\}$ & 
$A_1, C_3$\\
$\{B_1, Y_1\}$ & $C_1, A_3$ & & $\{C_1,C_2\}$ & $B_2, D$ & & $\{C_1,D\}$ & 
$C_2, C_3$\\
$\{C_1, X_1\}$ & $Y_1, Y_2$ & & $\{C_1,Y_1\}$ & $B_1, Z_1$ & & $\{C_1,Z_1\}$ & 
$Y_1, Z_2$\\
$\{D, X_1\}$ & $Y_1, Y_2$ & & $\{D,Y_1\}$ & $C_1, X_3$ & & $\{X_1,X_2\}$ & 
$Y_2, X_3$\\
$\{X_1, Y_1\}$ & $D, Z_1$ & & $\{X_1,Z_1\}$ & $Y_1, Z_2$ & & $\{Y_1,Z_1\}$ & 
$C_1, A_3$\\
$\{Z_1, Z_2\}$ & $A_1, X_1$ & & & & & & \\
\end{tabular}
\end{center} 
\end{proof}

\subsection{Sketch of non-collapsibility}
The purpose of this subsection is to give a rough idea why the complex $\L$
should not be 2-collapsible. This description could be useful, for instance,
for generalizations. However, the reader can easily skip this part. The author still prefer to include this discussion in order to explain how the complex is built up.

Let us split the collection $\LL$ into two parts $\LL^+ := \A \cup \B \cup \C
\cup \D$ and $\LL^- := \X \cup \Y \cup \Z$. The nerve of $\LL^+$, resp.
$\LL^-$, is denoted by $\L^+$, resp. $\L^-$. Both $\L^+$ and $\L^-$ are
triangulations of a disc with only three boundary edges $\{A_1,A_2\}$, $\{A_1,
A_3\}$, and $\{A_2,A_3\}$;
resp. $\{Z_1,Z_2\}$, $\{Z_1, Z_3\}$, and $\{Z_2,Z_3\}$; see
Figure~\ref{f:triang}. Only these boundary faces are 2-collapsible faces of
$\L^+$, resp. $\L^-$.

\begin{figure}
\label{f:triang}
\begin{center}
\includegraphics{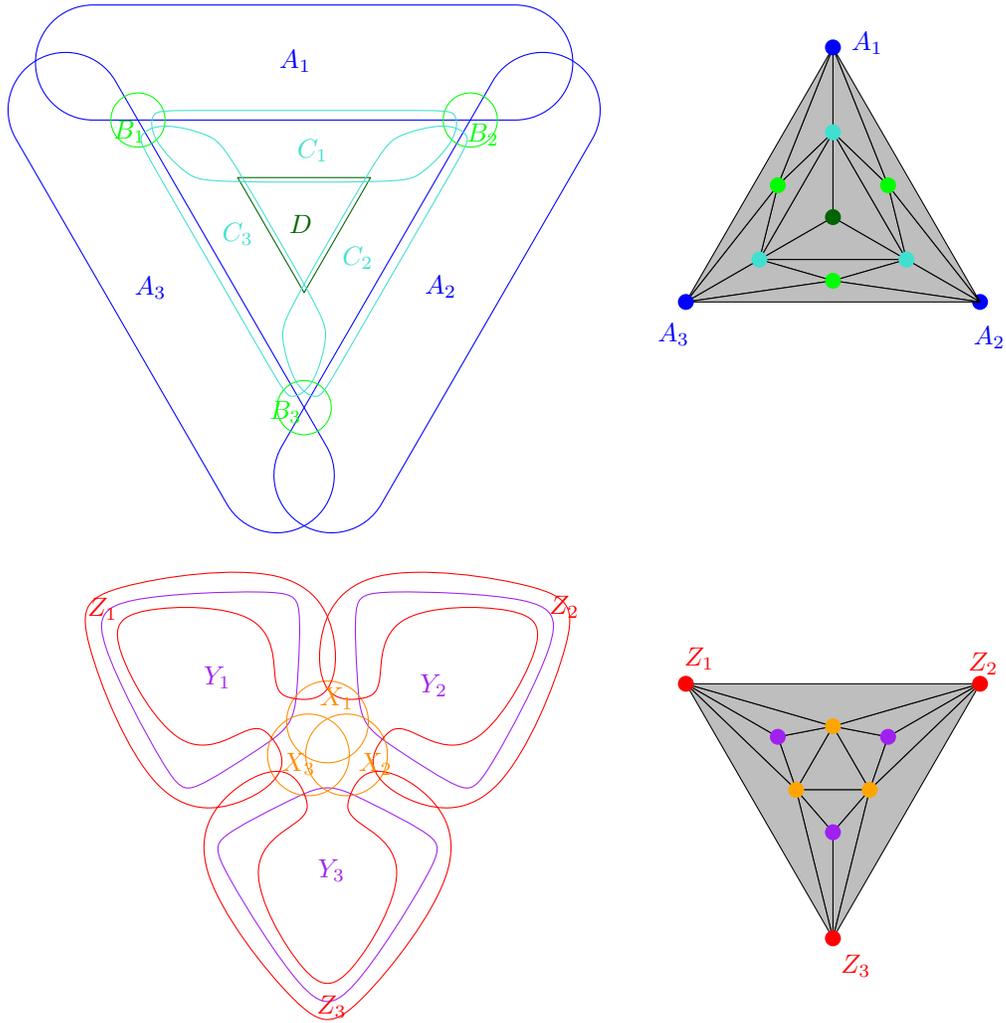}
\caption{The complexes $\L^+$ and $\L^-$.}
\end{center}
\end{figure}

By suitable overlapping of $\LL^+$ and $\LL^-$ (i.e., obtaining $\LL$) we get
that also the above mentioned boundary faces are not 2-collapsible anymore (in
whole $\L$). For instance $Z_1 \cap Z_2$ intersects $A_1$ (in addition to $X_1$
already in $\LL^-$); however, $A_1$ and $X_1$ are disjoint. Thus $\{Z_1, Z_2\}$
is not a 2-collapsible face of $\LL$.

It remains to check that merging $\LL^+$ and $\LL^-$ does not introduce any new
problems. It is, in fact, checked in a detail in the previous section. We just
mention that there is no problem with 1-faces which already appear in $\L^+$ or
$\L^-$. However; new 1-faces are introduced when one vertex comes from $\L^+$
and the second one from $\L^-$. For another triangulations these newly
introduced faces can be $2$-collapsible.\footnote{It would be perhaps possible
to show that the complex is not 2-collapsible even if the newly introduced
faces were 2-collapsible. Listing all 1-faces in the previous subsection seems,
however, more convenient for the current purpose.}

\section{Higher dimensions}
Joins of simplicial complexes will help us to generalize the counterexample to
higher dimensions. Let $\K$ and $\K'$ be simplicial complexes with the vertex
sets $V(\K)$ and $V(\K')$. Their \emph{join} is a simplicial complex $\K \star
\K'$ whose vertex set is the disjoint union $V(\K) \sqcup V(\K')$;\footnote{If
$A$ and $B$ are sets with $A \cap B \neq \emptyset$ then their disjoint union
can be defined as $A \sqcup B := A\times\{1\} \cup B \times \{2\}$.} and whose
set of faces is $\{\alpha \sqcup \beta: \alpha \in \K,
\beta \in \K'\}$.

We need the following two lemmas.

\begin{lemma}[{\cite[Lemma~4.2]{matousek-tancer09}}]
\label{l:gamma}
For every two simplicial complexes $\K$, $\K'$ we have $\gamma_0 (\K \star
\K') = \gamma_0 (\K) + \gamma_0(\K')$.
\end{lemma}

\begin{lemma}
\label{l:joinrep}
Let $\K$ be a convexly/topologically $d$-representable complex and $\K'$ be a
convexly/topologically $d'$-representable complex. Then $\K \star \K'$ is a
convexly/topologically $(d + d')$-representable complex.
\end{lemma}

\begin{proof}
Let $\F$ be a collection of convex sets/good cover in $\R^d$
such that $\K$ is isomorphic to the nerve of $\F$. Similarly $\F'$ is a
suitable collection in $\R^{d'}$ such that $\K'$ is isomorphic to the nerve of
$\F'$.

Let us set
$$
\F \star \F' := \{F \times \R^{d'}: F \in \F\} \cup \{\R^d \times F': F' \in
\F'\}.
$$
Then it is easy to check that $\K \star \K'$ is isomorphic to the nerve of $\F
\star \F'$. Moreover $\F \star \F'$ is a collection of convex sets/good cover
in $\R^{d + d'}$.
\end{proof}

Now we can finish the proof of our main result.
\begin{proof}[Proof of Theorem~\ref{t:main}.]
Let $\T$ be the simplicial complex consisting of two isolated points. The
complex $\T$ is topologically 1-representable and $\gamma_0(\T) = 1$.
Let us set
$$
\J = \L \star \underbrace{\T \star \cdots \star \T}_{d-2}.
$$
In topology, the complex $\J$ would be called $(d-2)$-tuple \emph{suspension}
of $\L$.
Then $\gamma_0 (\J) = d+1$ due to Proposition~\ref{p:gamma} and
Lemma~\ref{l:gamma}. On the other hand, $\J$ is topologically $d$-representable
due to Lemma~\ref{l:joinrep}.
\end{proof}

\section{Conclusion}

In the spirit of Helly-type theorems we could ask whether there is at least
some weaker bound for collapsibility of topologically $d$-representable
complexes.

\begin{question}
For which $d \geq 2$ there is a $d' \in \N$ (as least as possible) such that every topologically $d$-representable complex is $d$-collapsible?
\end{question}

Using joins of multiple copies of $\L$ (instead of suspensions of $\L$) we
obtain the following bound.

\begin{proposition}
For every $d \geq 2$ there is a simplicial complex which is topologically
$2d$-representable but not $(3d-1)$-collapsible.
\end{proposition}

\begin{proof}
Consider the complex $\underbrace{\L \star \cdots \star \L}_{d}$.
\end{proof}

If there is a wider gap among these notions it will also reflect at the gap
between $d$-representable and \emph{$d$-Leray} complexes obtained (with a
similar method) by Matou\v{s}ek and the author~\cite{matousek-tancer09}.

\section*{Acknowledgement}
I would like to thank Xavier Goaoc and Ji\v{r}\'{\i} Matou\v{s}ek for fruitful
discussions on this topic.

\bibliographystyle{alpha}
\bibliography{/home/martin/clanky/bib/grph,/home/martin/clanky/bib/topocom,/home/martin/clanky/bib/combgeo,/home/martin/clanky/bib/compl}

\end{document}